\documentclass[a4paper,11pt]{amsart}
\usepackage{amsfonts}
\usepackage{amsthm}
\usepackage{amssymb}
\usepackage{amsmath}
\usepackage{comment}
\usepackage{a4wide}
\usepackage{dsfont}
\usepackage{bbm}
\usepackage{xparse,mathrsfs,mathtools}
\usepackage{enumitem}
\usepackage{color}

\newtheorem{thm}{Theorem}

\newtheorem{cor}[thm]{Corollary}

\allowdisplaybreaks

\begin{document}

\begin{titlepage}

\title[Moment generating functions for lacunary sums]{Moment generating functions and moderate deviation principles for lacunary sums}

 \bigskip
\author{Christoph Aistleitner}
\author{Lorenz Fr\"uhwirth}
\author{Manuel Hauke}
\author{Maryna Manskova}


\begin{abstract} 
In a recent paper, Aistleitner, Gantert, Kabluchko, Prochno and Ramanan studied large deviation principles (LDPs) for lacunary trigonometric sums $\sum_{k=1}^N \cos(2 \pi n_k x)$, where the sequence $(n_k)_{k \geq 1}$ satisfies the Hadamard gap condition $n_{k+1} / n_k \geq q > 1$ for $k \geq 1$. A crucial ingredient in their work were asymptotic estimates for the moment generating function (MGF) of such sums, which turned out to depend on the fine arithmetic structure of the sequence $(n_k)_{k \geq 1}$ in an intricate way. In the present paper we carry out a detailed study of the MGF for lacunary trigonometric sums (without any structural assumptions on the underlying sequence, other than lacunarity), and we determine the sharp threshold where arithmetic effects start to play a role. As an application, we prove moderate deviation principles for lacunary trigonometric sums, and show that the tail probabilities are in accordance with Gaussian behavior throughout the whole range between the central limit theorem and the LDP regime.
\end{abstract}
	
\maketitle
\noindent
\address{Christoph Aistleitner: Graz University of Technology, Institute of Analysis and Number Theory, Steyrergasse 30, 8010 Graz, Austria}\\
\email{Mail: aistleitner@math.tugraz.at}\\

\noindent
\address{Lorenz Frühwirth: University of Passau, Faculty of Computer Science and Mathematics, Dr.-Hans-Kapfinger-Stra{\ss}e 30, 94032 Passau, Germany}\\
\email{Mail: lorenz.fruehwirth@uni-passau.de}\\

\noindent
\address{Manuel Hauke: Norwegian University of Science and Technology, Department of Mathematical Sciences, Sentralbygg 2 Gl{\o}shaugen, Trondheim, Norway}\\
\email{Mail: manuel.hauke@gmail.com}\\

\noindent
\address{Maryna Manskova: Graz University of Technology, Institute of Analysis and Number Theory, Steyrergasse 30, 8010 Graz, Austria}\\
\email{Mail: maryna.manskova@tugraz.at}\\

\section*{Acknowledgments}

CA is supported by the Austrian Science Fund (FWF), projects 10.55776/I4945, 10.55776/I5554, 10.55776/P34763 and 10.55776/P35322. MM is supported by the Austrian Science Fund (FWF), project DOC-183. 

\end{titlepage}

\section{Introduction and main result}

Lacunary trigonometric sums have a long and fascinating history. They are well-known for exhibiting properties which are fundamentally different from those of general trigonometric sums, and rather coincide in many cases with the behavior which is typically also observed for sums of independent random variables. Early examples of this phenomenon include Weierstra{\ss}' construction of a continuous, nowhere differentiable function, and Fatou's work on the almost everywhere convergence of lacunary trigonometric series, a topic which was brought to a culmination point by Kolmogorov \cite{kolm}. Among the most remarkable early papers on lacunary sums, from a probabilistic perspective, are Salem and Zygmund's \cite{salz} central limit theorem for lacunary trigonometric sums, Erd\H os and G\'al's \cite{erdg} law of the iterated logarithm, and Kac' work \cite{kac} on interference phenomena.\\

The most classical lacunary growth condition, called the Hadamard gap condition, requests that there exists a constant $q>1$ such that 
\begin{equation} \label{had}
\frac{n_{k+1}}{n_k}\geq q, \qquad k \geq 1,  
\end{equation}
and for such a sequence of frequencies one studies the sums
$$
\sum_{k=1}^N \cos(2 \pi n_k x) \qquad \text{and} \qquad \sum_{k=1}^N \sin(2 \pi n_k x)
$$
(which behave very similarly in most regards, so that often the results and proofs are only stated for the cosine sums). Among possible generalizations, all of which have been studied extensively, are: the introduction of a coefficient sequence $(c_k)_{k \geq 1}$, so that the sum becomes $\sum_{k=1}^N c_k \cos (2 \pi n_k x)$, see for example Weiss \cite{weiss} and Fr\"uhwirth and Hauke \cite{frueh}; replacing the Hadamard gap condition by a weaker growth condition, see for example Erd\H{o}s \cite{erd}, Takahashi \cite{taka}, Berkes \cite{berkes}, and Bobkov and G\"otze \cite{bob}; replacing $\cos(2 \pi x)$ by a more general 1-periodic function $f(x)$, see for example Kac \cite{kac}, Gapo\v{s}kin \cite{gapo}, and Aistleitner and Berkes \cite{ab}; replacing the trigonometric system by another orthogonal system such as the Walsh system, see for example Levizov \cite{levizov} and Berkes and Philipp \cite{bp}; taking a perspective from ergodic theory, see for example Conze and Le Borgne \cite{clb} and Conze, Le Borgne and Roger \cite{clbr}; and interpreting lacunary sums in the wider framework of the ``subsequence principle'' of probability theory and functional analysis, see Aldous \cite{ald1,ald2} and Berkes and P\'{e}ter \cite{bepe}. For more general background on lacunary sums and related topics, we refer to the recent survey article \cite{abt}.\\

It is notable that the central limit theorem and the law of the iterated logarithm for lacunary trigonometric sums hold in a universal form: if $(n_k)_{k \geq 1}$ satisfies \eqref{had}, then for all $t \in \mathbb{R}$
$$
\mathbb{P} \left( x \in (0,1):~ \frac{1}{\sqrt{N}} \sum_{k=1}^N \sqrt{2} \cos(2 \pi n_k x) \leq t \right) \to \Phi(t) \qquad \text{as $N \to \infty$,}
$$
where $\mathbb{P}$ is the Lebesgue measure and $\Phi$ is the distribution function of the standard normal distribution. Further,
$$
\limsup_{N \to \infty} \frac{\sum_{k=1}^N \sqrt{2} \cos(2 \pi n_k x)}{\sqrt{2 N \log \log N}} = 1 \qquad \text{a.e.}; 
$$
the normalization factor $\sqrt{2}$, which appears in both formulas, comes from the standard deviation of the identically distributed (but not independent) random variables $\cos(2\pi n_k x),~k \geq 1,$ which is
$$
\left( \int_0^1 \cos(2 \pi n_k x)^2 ~\mathrm{d}x \right)^{1/2} = \frac{1}{\sqrt{2}}.
$$
We emphasize again that neither formula depends in any way on specific (for example, number-theoretic) properties of the sequence $(n_k)_{k \geq 1}$, other than the assumption that it satisfies the growth condition \eqref{had}. This ``universal'' behavior at the scales of central limit theorem and law of the iterated logarithm changes dramatically when $\cos(2 \pi x)$ is replaced by a more general 1-periodic function $f(x)$, and many papers have been devoted to the intriguing effect that arithmetic properties of $(n_k)_{k \geq 1}$ have on the probabilistic behavior of $\sum f(n_k x)$ for lacunary $(n_k)_{k \geq 1}$; however, this is not the direction which we will pursue here, since throughout this paper we will stick to the case of pure trigonometric sums.\\

A very interesting effect was observed in a recent paper of Aistleitner, Gantert, Kabluchko, Prochno and Ramanan \cite{agkpr}. They studied so-called large deviation principles (LDPs) for lacunary trigonometric sums, which are essentially asymptotic estimates for the ``large deviation'' probabilities 
$$
\mathbb{P} \left( x \in (0,1):~ \frac{1}{N} \sum_{k=1}^N \sqrt{2} \cos(2 \pi n_k x) > t \right)
$$
(note that the scaling factor here is $1/N$, in contrast to the $1/\sqrt{N}$ of the central limit theorem). A comparison with the tail probabilities of the normal distribution suggests that these probabilities should be of order roughly $e^{-N t^2/2}$, so it makes sense to consider the ``rate function''
$$
I(t) := -  \lim_{N \to \infty} \frac{1}{N} \log \left(\mathbb{P} \left( x \in (0,1):~ \frac{1}{N} \sum_{k=1}^N \cos(2 \pi n_k x) > t \right)\right),
$$
provided such a limit actually exists. By the G\"artner--Ellis theorem, the calculation of such rate functions is essentially reduced to the calculation of the moment generating function (MGF), or more precisely, the calculation of the limit as $N \to \infty$ of
$$
\Lambda_N (\lambda) := \frac{1}{N} \log \left( \mathbb{E} \left( \exp \left( \lambda \sum_{k=1}^N \sqrt{2} \cos(2 \pi n_k x) \right) \right) \right),
$$
provided this limit exists (throughout this paper, expected values are always understood to be taken with respect to the Lebesgue measure on $(0,1)$). Now, if the system $(\cos(2 \pi n_k x))_{k \geq 1}$ were to exhibit truly independent behavior, then one could calculate
\begin{align*}
\mathbb{E} \left( \exp \left( \lambda \sum_{k=1}^N \sqrt{2} \cos(2 \pi n_k x) \right) \right) = &\prod_{k=1}^N \mathbb{E} \left( \exp \left( \lambda \sqrt{2} \cos(2 \pi n_k x) \right) \right) \\
= & \left( \mathbb{E} \left( \exp \left( \lambda \sqrt{2} \cos(2 \pi x) \right) \right) \right)^N,
\end{align*}
and accordingly one could expect the limit $\Lambda(\lambda)$ of $\Lambda_N(\lambda)$ to be
\begin{align}
\log \left( \int_0^1 e^{\lambda \sqrt{2} \cos(2 \pi x)} ~\mathrm{d}x \right) = & \log \left(I_0 \left(\sqrt{2} \lambda \right)\right) \nonumber\\
= & \log \left( \sum_{m=0}^\infty \frac{\lambda^{2m}}{2^m (m!)^2} \right) \nonumber\\
= & \frac{\lambda^2}{2} - \frac{\lambda^4}{16} + \frac{\lambda^6}{72} - +\dots, \label{series}
\end{align}
where $I_0$ denotes a modified Bessel function of the first kind. Indeed, in \cite{agkpr} it is proved that under the stronger gap condition $\frac{n_{k+1}}{n_k} \to \infty$ as $k \to \infty$, the limiting function $\Lambda(\lambda)$ exists and indeed equals $I_0 \left(\sqrt{2} \lambda \right)$, in accordance with true independent behavior. However, for other sequences $(n_k)_{k \geq 1}$ (of actual exponential growth, and not super-exponential growth), the MGF of $\sum_{k=1}^N \sqrt{2} \cos(2 \pi n_k x)$ could only be dealt with in \cite{agkpr} with the necessary precision under strong structural assumptions on $(n_k)_{k \geq 1}$, such as assuming the very special (``stationary'') situation $n_k = a^k$ for $k\geq 1$, for some integer $a \geq 2$.\\

In this paper, we study $\Lambda_N (\lambda)$ as a function of $\lambda \in [-1,1]$ for general lacunary sequences that satisfy \eqref{had}. We prove (Theorem \ref{th1}) that $\Lambda_N(\lambda) = \frac{\lambda^2}{2} + \mathcal{O}(\lambda^3)$ as $N \to \infty$ throughout the range $\lambda \in [-1,1]$ for all lacunary sequences $(n_k)_{k \geq 1}$. Thus the quadratic term in the expansion of $\Lambda_N(\lambda)$ is always (asymptotically) in accordance with the random model, and also in accordance with Gaussian behavior. However, the cubic term does not need to be in agreement with the random model, and (as the examples after the statement of the theorem will show) our error term $\mathcal{O}(\lambda^3)$ is optimal as far as general lacunary sequences are concerned. 

\begin{thm} 
\label{th1}
 Let $(n_k)_{k \geq 1}$ be a sequence of positive integers satisfying the Hadamard gap condition $\frac{n_{k+1}}{n_k} \geq q$, $k \geq 1$, for some constant $q > 1$. Then for all $\lambda \in [-1,1]$ we have
 $$
 \int_0^1 \exp \left( \lambda \sum_{k=1}^N \sqrt{2}\cos(2 \pi n_k x) \right) \mathrm{d}x = \exp \left( \frac{\lambda^2N}{2} + \mathcal{O} \left(\lambda^{3} N \right) \right)
 $$
 as $N \to \infty$, where the implied constant depends only on $q$.
\end{thm}

We give a few examples, which show that the term $\mathcal{O}(\lambda^3)$ in the statement of the theorem is optimal.\\

\begin{itemize}
    \item The sequence $n_k = 2^k,~k \geq 1$. This is the sequence which was already studied by Kac \cite{kac}. Note that this is a structurally very special case, since it has a very clear interpretation in terms of ergodic theory, where it can be seen as the action of the doubling map $T:~x \mapsto 2x$ mod $1$. In this case a limit function $\Lambda (\lambda)  = \lim_{N \to \infty} \Lambda_N (\lambda)$ indeed exists, and is calculated (using combinatorial tools) in the appendix of \cite{agkpr}, where it is shown to be 
    \begin{equation*} \label{2_function}
    \Lambda (\lambda) = \frac{\lambda^2}{2} + \frac{\lambda^3}{2 \sqrt{2}} + \frac{3 \lambda^4}{16} + \dots.
    \end{equation*}
    Note the presence of a cubic term, which is absent in the series \eqref{series} for the random model.\\
    
    \item We remark that the calculation of $\Lambda(\lambda)$ for $n_k=2^k,~k \geq 1$ is somewhat involved, since the system $\cos(2 \pi 2^k x))_{k \geq 1}$ has a rather complicated dependence structure. A simpler model is $n_1 = 1$ and 
    $$
    n_{k+1} = 2n_k \text{ for $k$ odd,}\qquad n_{k+1} = k! n_k \text{ for $k$ even}.
    $$
    Here, the factor $k!$ is quite arbitrarily chosen, since it is only used to break the overall dependence structure into pairs of consecutive indices. Using the fact that $n_{2k+1} / n_{2k}$ increases very quickly, one can show for this sequence that
    $$
    \mathbb{E} \left( \exp \left( \lambda \sum_{k=1}^N \sqrt{2} \cos(2 \pi n_k x) \right) \right) \approx \left(\mathbb{E} \left( \exp \left( \lambda \sqrt{2} \left( \cos(2 \pi x) + \cos(4 \pi x) \right) \right) \right)\right)^{N/2}
    $$
    as $N \to \infty$, and thus
    $$
    \Lambda(\lambda) = \frac{1}{2} \log \left( \int_0^1 \exp \left( \lambda \sqrt{2} \left( \cos(2 \pi x) + \cos(4 \pi x) \right) \right) \mathrm{d}x \right) = \frac{\lambda^2}{2} + \frac{\lambda^3}{4 \sqrt{2}} - \frac{\lambda^4}{16} + \dots
    $$
    Note again the presence of the cubic term.\\

    \item Another simple model is the sequence with $n_1 = 1$, and
    $$
   n_{k+1} = 2 n_k \quad \text{and} \quad  n_{k+2} = 3 n_{k}  \qquad \text{for $k \equiv 1$ mod 3},
    $$
    and $n_{k} = k! n_{k-1}$ when $k \equiv 0$ mod 3. The dependence structure of this sequence decomposes into triples of consecutive indices, and one has 
    \begin{align*}
    & \mathbb{E} \left( \exp \left( \lambda \sum_{k=1}^N \sqrt{2} \cos(2 \pi n_k x) \right) \right) \\
    \approx & \left(\mathbb{E} \left( \exp \left( \lambda \sqrt{2} \left( \cos(2 \pi x) + \cos(4 \pi x) + \cos(6 \pi x) \right) \right) \right)\right)^{N/3}
    \end{align*}
    as $N \to \infty$, and thus
    \begin{align*}
    \Lambda(\lambda) = & \frac{1}{3} \log \left( \int_0^1 \exp \left( \lambda \sqrt{2} \left( \cos(2 \pi x) + \cos(4 \pi x) + \cos(6 \pi x)\right) \right) \mathrm{d}x \right) \\
    = &\frac{\lambda^2}{2} + \frac{\lambda^3}{2\sqrt{2}} + \frac{7 \lambda^4}{144} + \dots
    \end{align*}
    Note again the presence of the cubic term, with yet another coefficient. We further remark that, from a perspective of Diophantine equations (which is from a technical point of view the key to understanding lacunary sums), this sequence is rather different from the two sequences mentioned before. In the above cases, the sequences were constructed in such a way that there are many solutions $k$ of the equation $n_{k+1} = 2 n_k$, and the reason why these solutions contribute to the cubic term in the expansion of $\lambda$ is that this 2-term Diophantine equation has coefficients $1$ and $2$, such that the sum of coefficients is $1+2=3$, with 3 meaning ``cubic'' (the way how the coefficients of these Diophantine equations arise from powers of $\lambda \cos(2 \pi n_k x)$, and the total power of $\lambda$ is thus related to the sum of the coefficients of the Diophantine equations, will become clear during the proof of Theorem \ref{th1}). Now in this new example, the sequence is constructed in such a way that there are many solutions $k$ of $n_{k+1} = n_k + n_{k-1}$, which is a 3-term Diophantine equation with coefficients $1,1,1$, leading to a non-vanishing cubic term in the expansion of $\Lambda(\lambda)$ since 1+1+1=3. The sequence in this last example is modeled on the Fibonacci sequence $(F_k)_{k \geq 1}$, which of course also satisfies $F_{k+2} = F_{k+1} + F_{k}$ (for all $k$, not only for those congruent 0 mod 3). One could probably also study the Fibonacci sequence itself, rather than the sequence $(n_k)_{k \geq 1}$ as constructed above, and prove that $\Lambda_N(\lambda)$ converges to some limiting function $\Lambda(\lambda)$ with non-vanishing cubic term in its series expansion, but the dependence structure of $(\cos(2\pi F_k x))_{k \geq 1}$ is much more delicate than that of our construction, and it is a priori not even clear that $\Lambda_N(\lambda)$ converges for $(F_k)_{k \geq 1}$. Note that in contrast to $(\cos(2 \pi 2^k x))_{k \geq 1}$, the system $(\cos(2 \pi F_k x))_{k \geq 1}$ does not seem to have an immediate interpretation from the perspective of ergodic theory.\\
    
    \item From the discussion above, it is not difficult to construct a lacunary sequence $(n_k)_{k \geq 1}$ for which $\Lambda_N(\lambda)$ fluctuates and does not converge as $N \to \infty$ (beyond the quadratic term in the series expansion). Such a sequence could for example be constructed by ``mixing'' finite segments (of increasing length) of the sequences from the previous examples. Together with the examples above, this shows that indeed $\Lambda_N(\lambda) = \lambda^2/2 + \mathcal{O}(\lambda^3)$ is the best result one could possibly hope for, for general lacunary sequences $(n_k)_{k \geq 1}$.\\
\end{itemize}

We note in passing that our method of proof entails a potential dependence on the growth factor $q$ for the implied constant in the $\mathcal{O}(\lambda^3)$ term, in the sense that this implied constant possibly needs to be increased as $q \to 1$. It seems plausible that this $q$-dependent implied constant could actually be replaced by an absolute constant, the optimal value for such a constant possibly being $1/\sqrt{2}$ (arising from the case of the Fibonacci sequence). However, we have not pursued this topic any further.\\
\par{}

Our ability to provide optimal estimates for the (asymptotic) MGF of $\sum_{k =1}^N \sqrt{2} \cos(2 \pi n_k x)$ allows us to establish moderate deviation principles (MDPs) (see Corollary \ref{cor:MDP}) for the latter via the Gärtner-Ellis theorem (see, e.g., \cite[Theorem 2.3.6]{DZLDPs}). These principles, originating from \cite{RJMDP}, have become an important concept in the study of the distribution of sequences of random variables. Let $\lambda = \lambda_N$ be a null sequence with $ \lambda \sqrt{N} \to \infty$ and let $(X_N)_{N \in \mathbb{N}}$ be a sequence of random variables defined on $(0,1)$. We say that $\left( X_N \right)_{N \in \mathbb{N}}$ satisfies an MDP at speed $ 1/\lambda^2 $ if there exists $\sigma^2 > 0$ such that for all Borel-sets $B \subseteq \mathbb{R}$ we have
\begin{equation}
\label{eq:MDP}
- \inf_{x \in B^{\circ}} \frac{x^2}{2 \sigma^2} \leq \liminf_{N \rightarrow \infty} \lambda^2  \log \mathbb{P} \left( X_N \in B \right) \leq \limsup_{N \rightarrow \infty}  \lambda^2  \log \mathbb{P} \left( X_N \in B \right) \leq - \inf_{x \in \overline{B}} \frac{x^2}{2 \sigma^2}, 
\end{equation}
where $B^{\circ}$ is the interior and $\overline{B}$ is the closure of $B$. If one considers a sum of independent random variables $Y_1 + \ldots + Y_N$, all having mean $0$ and variance $\sigma^2$, it is well-known that $ X_N:= \frac{\lambda}{\sqrt{N}} \sum_{k=1}^N Y_k$ satisfies an MDP at speed $1/\lambda^2 $ with variance $\sigma^2$ (see, e.g., \cite[Theorem 3.7.1]{DZLDPs}). Heuristically speaking, for such a sequence, the central limit theorem would suggest 
\[
\mathbb{P} \left( \frac{\lambda}{\sqrt{N}} \sum_{k=1}^N Y_k \geq x \right) \approx e^{- \frac{x^2}{2 (\lambda \sigma)^2}}, \quad \text{as} \quad N \longrightarrow \infty.
\]
An MDP tells us that, on a logarithmic scale, the above holds even for arbitrary Borel sets instead of intervals of the form $[x, \infty)$ for some $x > 0$. 
In the following corollary we use Theorem \ref{th1} to obtain moderate deviation principles for arbitrary trigonometric lacunary sums.

\begin{cor}
\label{cor:MDP}
Let $(n_k)_{k \in \mathbb{N}}$ be a lacunary sequence satisfying \eqref{had}. Then, for any positive null sequence $\lambda = \lambda_N$ with $\sqrt{N} \lambda \to \infty$,
\[
\frac{\lambda}{\sqrt{N}} \sum_{k=1}^N \sqrt{2} \cos(2 \pi n_k x)
\]
satisfies an MDP at speed $1/\lambda^2$ in the sense of \eqref{eq:MDP} with $\sigma^2=1$.
\end{cor}

One should compare Corollary \ref{cor:MDP}, which gives a clean universal result, to the results in \cite{PSMDP}, where MDPs for lacunary sums $\sum f(n_k x)$ generated by trigonometric polynomials $f$ were studied. There, the authors had to impose strong arithmetic assumptions on the lacunary sequence $(n_k)_{k \geq 1}$, as well as an additional growth condition upon $\lambda$, in order to obtain their moderate deviation principles (see \cite[Theorem 1]{PSMDP}). It is not surprising that stronger assumptions are necessary to obtain MDPs in the general case $\sum f(n_k x)$, since CLT and LIL also hold in the case of $\sum f(n_k x)$ only under additional arithmetic assumptions. However, \cite{PSMDP} provides only sufficient criteria for the validity of an MDP, while sharp (necessary and sufficient) arithmetic conditions are known in the CLT and LIL cases, see \cite{ab,afp}.
 
\section{Proofs}

\begin{proof}[Proof of Theorem \ref{th1}]

Assume that $N$ is ``large'', and let $\lambda \in \mathbb{R}$ be given. Let $L$ (``long'') and $s$ (``short'') be two positive integers, to be specified later. We decompose the set of indices $\{1,2,\dots, N\}$ into a disjoint union of blocks
\begin{equation}\label{decomp}
\{1, \dots, N\} = \Delta_1\cup\Delta_1^\prime\cup\Delta_2\cup\Delta_2^\prime\cup \cdots \cup \Delta_M\cup\Delta_M^\prime,
\end{equation}
such that $|\Delta_i|=L$ and $|\Delta_i^\prime|=s$ for $1\leq i\leq M$, where $M = N / (L + s)$. To simplify notations we assume that indeed $N$ is divisible by $s+L$ without remainder; otherwise one has to shorten one of the last two blocks and possibly skip $\Delta_M^\prime$, which does not affect the overall calculation. The decomposition in \eqref{decomp} has to be understood in the sense that 
$$
\Delta_1 < \Delta_1^\prime < \Delta_2 < \Delta_2^\prime < \dots 
$$
holds element-wise. The sum over all the long blocks $\Delta_i$ will give the main contribution, and the role of the short blocks $\Delta_i^\prime$ is to ensure that the frequencies $n_k$ and $n_\ell$ of two trigonometric functions $\cos(2 \pi n_k x)$ and $\cos(2 \pi n_\ell x)$ with indices $k$ and $\ell$ coming from two {\em different} long blocks are of very different size, thus creating a high degree of ``independence'' between trigonometric functions (interpreted as random variables) whenever the respective indices are in different long blocks.\\

Recall that $q>1$ is the growth factor of the lacunary sequence under consideration. We define $s$ as the smallest positive integer for which 
\begin{equation} \label{s_size}
q^s > \left(1-q^{-1/2}\right)^{-1} \quad \text{and} \quad 1 + 4 q^{-s} \leq \sqrt{q}.
\end{equation}
The precise nature of these assumptions in not important for the moment (they will be required at different places in the proof), but we remark that they can all be satisfied by choosing $s$ sufficiently large in terms of $q$. Note in particular that $s$ clearly does not depend on $N$ or $\lambda$. Furthermore, we set $L= \left\lceil \frac{1}{2 |\lambda|} \right\rceil$. Note that the conclusion of the theorem is trivial unless $|\lambda| \to 0$ as $N \to \infty$. Thus we can assume, throughout the rest of the paper, that $N$ and $\lambda$ are such that
\begin{equation} \label{sL_assumptions}
s < L \qquad \text{and} \qquad L \leq \frac{1}{\sqrt{2} |\lambda|}.
\end{equation}

We have
 \begin{equation}\begin{split}
 &  \int_0^1 \exp \left( \lambda \sum_{k=1}^N \sqrt{2}\cos(2 \pi n_k x) \right) \mathrm{d}x \\
 = & \int_0^1 \prod_{i=1}^M \left(\exp \left( \lambda \sum_{k\in \Delta_i} \sqrt{2}\cos(2 \pi n_k x) \right) \cdot \exp \left( \lambda \sum_{k\in \Delta_i^\prime}\sqrt{{2}} \cos(2 \pi n_k x) \right) \right) \mathrm{d}x, \label{hölder_this}
 \end{split}
 \end{equation}
and the aim now is to estimate this integral of a product in terms of a product of integrals (note that under actual stochastic independence, the integral of a product would automatically decompose into a product of integrals, so what we are trying to achieve here is to mimic such independent behavior).\\ 

We start by giving an upper bound for the integral in \eqref{hölder_this}. One can easily show that 
\begin{equation} \label{taylor_short}
\exp(x)\leq 1+x+x^2, \qquad \text{for $|x|\leq 1$,}
\end{equation}
as well as 
\begin{equation} \label{taylor_long}
\exp(x)\leq 1+x+\frac{x^2}{2}+\frac{x^3}{6} + x^4, \qquad \text{also for $|x|\leq 1$}.
\end{equation}

Since by \eqref{sL_assumptions} we have
\begin{align}\label{Taylor_long}
     \left|\lambda \sum_{k\in \Delta_i} \sqrt{2}\cos(2 \pi n_k x) \right|\leq  \sqrt{2} |\lambda | L = \sqrt{2}|\lambda| \frac{1}{\sqrt{2} |\lambda|}\leq 1,
\end{align} 
using \eqref{taylor_long} for a long block $\Delta_i$ for some $i$ we obtain
\begin{align}
& \exp \left(\lambda \sum_{k\in \Delta_i} \sqrt{2}\cos(2 \pi n_k x) \right) \nonumber\\
\leq&\, 1 +  \lambda \sum_{k\in \Delta_i} \sqrt{2}\cos(2 \pi n_k x)+\frac{\lambda^2}{2} \left(\sum_{k\in \Delta_i}\sqrt{2}\cos(2 \pi n_k x)\right)^2  \label{expand_1}\\
&\phantom{\, 1 } + \frac{\lambda^3}{6} \left(\sum_{k\in \Delta_i}\sqrt{2}\cos(2 \pi n_k x)\right)^3 + \lambda^4 \left(\sum_{k\in \Delta_i}\sqrt{2}\cos(2 \pi n_k x)\right)^4, \label{expand_2}
\end{align}
and using \eqref{taylor_short} for a short block we similarly have
\begin{align}
& \exp \left(\lambda \sum_{k\in \Delta_i'} \sqrt{2}\cos(2 \pi n_k x) \right) \nonumber\\
\leq & 1 + \lambda \sum_{k\in \Delta_i'} \sqrt{2}\cos(2 \pi n_k x)+ \lambda^2 \left(\sum_{k\in \Delta_i'}\sqrt{2}\cos(2 \pi n_k x)\right)^2 \nonumber\\
\leq & 1 + \lambda \sum_{k\in \Delta_i'} \sqrt{2}\cos(2 \pi n_k x)+ 2\lambda^2  s^2. \label{only_kept}
\end{align}
Note that we are more precise with the approximation for the long blocks, and take a polynomial of relatively high degree to upper-bound the exponential function. In particular, the Taylor polynomial that we use for the long blocks needs to have the ``correct'' constant $1/2$ in the quadratic term, since this will show up as the main term $\exp(\lambda^2 N / 2)$ for the MGF in the final result. Furthermore, the polynomial for the long blocks needs to have degree at least three, since we want to end up with a cubic error term in the final result, and since the polynomial actually needs to be a pointwise upper bound for the exponential function, it turns out that we actually need to go up to degree four (i.e.\ even degree) with this polynomial. In contrast, the short blocks will only give a negligible contribution, and in \eqref{only_kept} only the linear terms remain as an actual function of $\cos(2 \pi n_k x)$ to keep the subsequent combinatorics as simple as possible. There is a very delicate balance between the lengths of the long and short blocks in the decomposition of the index set, and the degree of the polynomials which are used for each of the two respective cases.\\

We will expand the sums in \eqref{expand_1} and \eqref{expand_2}, keeping special track of the diagonal terms in the squared expression, for which we use $\cos(x)^2 = 1/2 + \cos(2 x)/2$. Hence we can rewrite the expression in lines  \eqref{expand_1} and \eqref{expand_2} as
\begin{align*}
S_i(x) := &\, 1 + \frac{\lambda^2 L}{2} +  \lambda \sum_{k \in \Delta_i} \sqrt{2}\cos(2 \pi n_k x) + \frac{\lambda^2}{2} \sum_{k \in \Delta_i} \cos(2 \pi 2 n_k x) \\
& \qquad + \lambda^2 \sum_{\substack{k_1,k_2 \in \Delta_i, \\ k_1 \neq k_2}} \cos(2 \pi n_{k_1} x) \cos(2\pi n_{k_2} x) \\
& \qquad + \frac{\lambda^3 2^{3/2}}{6} \sum_{k_1,k_2,k_3 \in \Delta_i} \cos(2 \pi n_{k_1} x) \cos(2 \pi n_{k_2} x) \cos(2 \pi n_{k_3} x)  \\
& \qquad + 4 \lambda^4  \sum_{k_1,k_2,k_3,k_4 \in \Delta_i} \cos(2 \pi n_{k_1} x) \cos(2 \pi n_{k_2} x) \cos(2 \pi n_{k_3} x) \cos(2 \pi n_{k_4} x).
\end{align*}
Recall that $\cos(x) \cos(y) = \frac{\cos(x+y)}{2} + \frac{\cos(x-y)}{2}$. Accordingly, all the products of cosine-functions which appear in the definition of $S_i(x)$ can be rewritten into cosines of sums/differences of frequencies such as $n_{k_1} \pm n_{k_2}$, $n_{k_1} \pm n_{k_2} \pm n_{k_3}$ or $n_{k_1} \pm n_{k_2} \pm n_{k_3} \pm n_{k_4}$. We will denote by $i^-$ and $i^+$ the smallest and the largest elements of $\Delta_i$, respectively. Then the sum $\sum_{k \in \Delta_i} \cos(2 \pi n_k x)$ in the definition of $S_i(x)$ is a trigonometric polynomial with frequencies between $n_{i^-}$ and $n_{i^+}$. Likewise, the sum $\sum_{k \in \Delta_i} \cos(2 \pi 2 n_k x)$ is a trigonometric polynomial with frequencies between $2 n_{i^-}$ and $2 n_{i^+}$. For the next term, we note that 
\begin{align*}
& \lambda^2 \sum_{\substack{k_1,k_2 \in \Delta_i, \\ k_1 \neq k_2}} \cos(2 \pi n_{k_1} x) \cos(2\pi n_{k_2} x) \\
= & 2 \lambda^2 \sum_{\substack{k_1,k_2 \in \Delta_i, \\ k_1 > k_2}} \left( \frac{\cos(2 \pi (n_{k_1} + n_{k_2})  x)}{2} + \frac{\cos(2 \pi (n_{k_1} - n_{k_2})  x)}{2} \right).
\end{align*}
Clearly, the value of $n_{k_1}+n_{k_2}$ is always between $2n_{i^-}$ and $2 n_{i^+}$. For the value of $n_{k_1} - n_{k_2}$, we need to be a bit more careful. When $k_1 > k_2$, then 
$$
n_{k_1} - n_{k_2} \geq (1 - q^{-1}) n_{k_1} \geq n_{i^-}, \qquad \text{provided that $k_1 - i^- \geq \log_q \left( \left( 1 - q^{-1}\right)^{-1} \right)$.}
$$
Accordingly, the number of pairs of indices $k_1,k_2 \in \Delta_i$ with $k_1 > k_2$, for which $n_{k_1} - n_{k_2}$ is smaller than $n_{i^-}$, is $\mathcal{O}(1)$. Thus there exists a trigonometric cosine-polynomial $T_i^{(2)}(x)$, all of whose frequencies lie between $n_{i^-}$ and $2 n_{i^+}$, such that
$$
\lambda^2 \sum_{\substack{k_1,k_2 \in \Delta_i, \\ k_1 \neq k_2}} \cos(2 \pi n_{k_1} x) \cos(2\pi n_{k_2} x) = \lambda^2 T_i^{(2)}(x) + \underbrace{\mathcal{O} (\lambda^2)}_{= \mathcal{O} (\lambda^3 L)}.
$$
This finishes our analysis of the contribution of the terms of second order to $S_i$. The contribution of the terms of third and fourth order is slightly more difficult to analyze.\\

Consider the term $\sum_{k_1,k_2,k_3 \in \Delta_i} \cos(2 \pi n_{k_1} x) \cos(2 \pi n_{k_2} x) \cos(2 \pi n_{k_3} x)$. The product of cosine-functions can be re-written as a linear combination of functions of the form $\cos(2\pi (\pm n_{k_1} \pm n_{k_2} \pm n_{k_3})x)$. If all three ``$\pm$'' signs are the same, then such a cosine has a frequency between $3 n_{i^-}$ and $3 n_{i^+}$. For other combinations of signs, we are interested in how often the frequency $|\pm n_{k_1} \pm n_{k_2} \pm n_{k_3}|$ can be ``small''. By symmetry, it suffices to count the number of such triples of indices in the case where the combinations of signs is ``$(+,-,-)$''. Thus, let us assume that 
\begin{equation} \label{equ_three}
\left| n_{k_1}- n_{k_2} - n_{k_3} \right| < n_{i^-}
\end{equation}
for some triple of indices $(k_1,k_2,k_3) \in \Delta_i^3$, and see how often this is possible. Assume w.l.o.g.\ that $k_2 \geq k_3$. If $k_2 \geq k_1$, then $n_{k_1}- n_{k_2} - n_{k_3} \leq -n_{k_3} \leq - n_{i^-}$, so in this case \eqref{equ_three} fails. Thus, we must have $k_1 > k_2$ (recall that we also assumed $k_2 \geq k_3$). There is only an absolutely bounded amount of possibilities for $(k_1,k_2,k_3) \in \Delta_i^3$ with $k_1 > k_2 \geq k_3$, for which $n_{k_1} \leq 2 n_{i^-}$. On the other hand, if $k_1$ is such that $n_{k_1} \geq 2 n_{i^-}$, then it is easily seen that $n_{k_2} \geq \frac{n_{k_1}}{4}$ is a necessary condition for \eqref{equ_three}, and by the lacunary growth condition there are at most $\mathcal{O}(1)$ many choices of $k_2$ with $k_2 < k_1$ for which this holds. 
Finally, for any given $k_1$ and $k_2$, there are at most $\mathcal{O}(1)$ many values of $k_3$ for which \eqref{equ_three} holds, again as a consequence of the lacunary growth condition. Overall, that means that the number of triples $(k_1,k_2,k_3) \in \Delta_i^3$ for which \eqref{equ_three} can hold is of cardinality at most $\mathcal{O}(L)$. In other words, we can write 
$$
\frac{\lambda^3 2^{3/2}}{6} \sum_{k_1,k_2,k_3 \in \Delta_i} \cos(2 \pi n_{k_1} x) \cos(2 \pi n_{k_2} x) \cos(2 \pi n_{k_3} x)  = \lambda^3 T_i^{(3)}(x) + \mathcal{O} (\lambda^3 L), 
$$
where $T_i^{(3)}(x)$ is a suitable trigonometric cosine-polynomial whose frequencies are all between $n_{i^-}$ and $3 n_{i^+}$.\\

Finally, we carry out a similar analysis for 
$$
\sum_{k_1,k_2,k_3,k_4 \in \Delta_i} \cos(2 \pi n_{k_1} x) \cos(2 \pi n_{k_2} x) \cos(2 \pi n_{k_3} x) \cos(2 \pi n_{k_4} x).
$$
This can be re-written into a sum of cosine-functions of the form $\cos (2\pi (\pm n_{k_1} \pm n_{k_2} \pm n_{k_3} \pm n_{k_4})x)$. If all four ``$\pm$'' signs are the same, then the frequency is between $4 n_{i^-}$ and $4 n_{i^+}$. The other two cases (up to symmetry) are those of signature $(+,+,-,-)$ and  of signature $(+,+,+,-)$.\\
Assume first that we study the case of signature $(+,+,-,-)$, and that for some $(k_1,k_2,k_3,k_4) \in \Delta_i^4$ we have
\begin{equation} \label{equ_four}
|n_{k_1} + n_{k_2} - n_{k_3} - n_{k_4} | < n_{i^-}. 
\end{equation}
Assume w.l.o.g.\ that $k_1 \geq k_2$, that $k_3 \geq k_4$, and that $k_1 \geq k_3$. There are only $\mathcal{O}(1)$ many 4-tuples $(k_1,k_2,k_3,k_4)$ satisfying these requirements for which $n_{k_1} \leq 2 n_{i^-}$. Thus we will assume in the sequel that $n_{k_1} \geq 2 n_{i^-}$. Then it is easily seen that $n_{k_3} \geq \frac{n_{k_1}}{4}$ is a necessary condition for \eqref{equ_four}, since otherwise
$$
n_{k_1} + n_{k_2} - n_{k_3} - n_{k_4} \geq n_{k_1} - 2 n_{k_3} \geq \frac{n_{k_1}}{2} \geq n_{i^-}.
$$
Note that for given $k_1$, there are only $\mathcal{O}(1)$ many choices of $k_3$ for which $k_3 \leq k_1$ and $n_{k_3} \geq \frac{n_{k_1}}{4}$. Now, for any given $k_1,k_2,k_3$, there are at most $\mathcal{O}(1)$ many choices of $k_4$ for which \eqref{equ_four} holds. Thus in summary, in \eqref{equ_four} at most $k_1,k_2$ can be freely chosen (which we can do in $L^2$ many ways), which then restricts $k_3,k_4$ to $\mathcal{O}(1)$ many possibilities, and so the number of 4-tuples $(k_1,k_2,k_3,k_3)$ for which \eqref{equ_four} holds is of order at most $\mathcal{O}(L^2)$. \\
Next we study the case of signature $(+,+,+,-)$, and assume that for some $(k_1,k_2,k_3,k_4) \in \Delta_i^4$ we have
\begin{equation} \label{equ_four_b}
|n_{k_1} + n_{k_2} + n_{k_3} - n_{k_4} | < n_{i^-}. 
\end{equation}
It is clear that this can only be true if $k_4 > \max(k_1,k_2,k_3)$. Assume w.l.o.g.\ that $k_3 \geq \max(k_1,k_2)$. There are $\mathcal{O}(1)$ many $(k_1, k_2,k_3,k_4) \in \Delta_i^4$ with $k_4 > \max(k_1,k_2,k_3)$ and simultaneously $n_{k_4} \leq 2 n_{i^-}$. Thus we will assume in the sequel that $n_{k_4} \geq 2n_{i^-}$. In this case it is easily seen that $n_{k_3} \geq \frac{n_{k_4}}{6}$ is necessary for \eqref{equ_four_b} to hold, as a consequence of $k_3 \geq \max(k_1,k_2)$. Recall that we must also have $k_3 < k_4$. Thus for given $k_4$, there are only $\mathcal{O}(1)$ many possible values for $k_3$. Furthermore, for any given configuration of $k_2,k_3,k_4$, there are at most $\mathcal{O}(1)$ many possible values for $k_1$ such that \eqref{equ_four_b} is satisfied. Thus the number of 4-tuples $(k_1,k_2,k_3,k_4) \in \Delta_i^4$ for which \eqref{equ_four_b} holds is again of order at most $\mathcal{O}(L^2)$. Consequently, we can write 
\begin{align*}
& 4 \lambda^4  \sum_{k_1,k_2,k_3,k_4 \in \Delta_i} \cos(2 \pi n_{k_1} x) \cos(2 \pi n_{k_2} x) \cos(2 \pi n_{k_3} x) \cos(2 \pi n_{k_4} x) \\
= & \lambda^4 T_i^{(4)}(x) + \underbrace{\mathcal{O}(\lambda^4 L^2)}_{= \mathcal{O}(\lambda^3 L)},
\end{align*}
where $T_i^{(4)}(x)$ is a trigonometric cosine-polynomial whose frequencies are all between $n_{i^-}$ and $4 n_{i^+}$.\\

Combining all our findings, we have shown that there exists a trigonometric cosine-polynomial \[T_i(x) := \lambda \sum_{k \in \Delta_i} \sqrt{2}\cos(2 \pi n_k x) + \frac{\lambda^2}{2} \sum_{k \in \Delta_i} \cos(2 \pi 2 n_k x) + \lambda^2 T_i^{(2)}(x) + \lambda^3 T_i^{(3)}(x) + \lambda^4 T_i^{(4)}(x),\] all of whose frequencies are between $n_{i^-}$ and $4 n_{i^+}$, such that
$$
S_i(x) = 1 + \frac{\lambda^2 L}{2} + T_i(x) + \mathcal{O}(\lambda^3 L).
$$
Thus, recalling \eqref{only_kept} we have
\begin{align}
& \exp \left(\lambda \sum_{k\in \Delta_i} \sqrt{2}\cos(2 \pi n_k x) \right)   \exp \left(\lambda \sum_{k\in \Delta_i'} \sqrt{2}\cos(2 \pi n_k x) \right) \nonumber\\
\leq & \left(1 + \frac{\lambda^2 L}{2} + T_i(x) + \mathcal{O}(\lambda^3 L) \right) \left(1 + \lambda \sum_{k\in \Delta_i'} \sqrt{2}\cos(2 \pi n_k x)+ 2\lambda^2  s^2 \right) \nonumber\\ 
\leq & 1 + \frac{\lambda^2 L}{2} + T_i(x) +  \lambda \sum_{k\in \Delta_i'} \sqrt{2}\cos(2 \pi n_k x) + T_i(x) \lambda \sum_{k\in \Delta_i'} \sqrt{2}\cos(2 \pi n_k x) + \mathcal{O}(\lambda^3 L), \label{continue_here}\textit{}
\end{align}
where we have also used that $S_i(x)$ and hence $T_i(x)$ is of order $\mathcal{O}(1)$.
We write
$$
T_i(x) = \sum_{m \geq 1} c_m \cos(2 \pi m x)
$$
for suitable coefficients $c_m$, and let $m$ be the frequency of a cosine-function that appears as a summand of $T_i(x)$ with a coefficient $c_m \neq 0$. Let $\ell \in \Delta_i'$ be fixed. We ask when it is possible that 
\begin{equation} \label{mell}
|m - n_\ell| \leq n_{i^-}.
\end{equation} 
For this we need to consider how $m$ can arise as a summand of $T_i(x)$. 
\begin{itemize}
\item First case: $m$ appears as a frequency in $\lambda \sum_{k \in \Delta_i} \sqrt{2} \cos(2 \pi n_k x)$ for some $k$. But then
$$
|m - n_\ell| \geq ( 1-q^{-1} ) n_\ell > n_{i^-},
$$ 
as a consequence of $\ell \geq i^+ + 1 \geq i^- + L \geq i^- + s$, so that \[\frac{n_\ell}{n_{i^-}} \geq q^s > (1-q^{-1/2})^{-1} \geq (1-q^{-1})^{-1}\] by the first inequality in \eqref{s_size}. So this case is actually impossible.
\item Second case: $m$ appears as a summand in $\frac{\lambda^2}{2} \sum_{k \in \Delta_i} \cos(2 \pi 2 n_k x)$. For fixed $\ell$ there are at most $\mathcal{O}(1)$ many $m$ satisfying \eqref{mell} that arise in this way, and all such contributions have a factor $\lambda^2$.
\item Third case: $m$ appears in $T_i^{(2)}$, where it arises as $m=n_{k_1} \pm n_{k_2}$. In either case of the sign, for fixed $k_1$, there are at most $\mathcal{O}(1)$ possibilities for $k_2$ such that \eqref{mell} is satisfied. Thus, there are at most $\mathcal{O}(L)$ many possibilities how $m$ can arise in this way, where each such occurrence comes with a factor $\lambda^2$.
\item Fourth case: $m$ appears in $T_i^{(3)}$, where it arises as $m = \pm n_{k_1} \pm n_{k_2} \pm n_{k_3}$. For a fixed configuration of $\pm$ signs, and for any choice of $k_1,k_2 \in \Delta_i$, there are at most $\mathcal{O}(1)$ many values of $k_3$ such that \eqref{mell} is satisfied. Thus overall there are at most $\mathcal{O}(L^2)$ many possible ways how $m$ can arise in this way, and each such occurrence has a factor $\lambda^3$. 
\item Fifth case: $m$ appears in $T_i^{(4)}$, where it arises as $m = \pm n_{k_1} \pm n_{k_2} \pm n_{k_3} \pm n_{k_4}$. For a fixed configuration of $\pm$ signs, and for any choice of $k_1,k_2,k_3 \in \Delta_i$, there are at most $\mathcal{O}(1)$ many values of $k_4$ such that \eqref{mell} satisfied. Thus overall there are at most $\mathcal{O}(L^3)$ many possible ways how $m$ can arise in this way, and each such occurrence has a factor $\lambda^4$. 
\end{itemize} 
We have 
\begin{align*}
& T_i(x)  \lambda \sum_{\ell \in \Delta_i'} \sqrt{2}\cos(2 \pi n_\ell x) \\   
= & \sum_{m \geq 1} c_m \cos(2 \pi m x) \lambda \sum_{\ell\in \Delta_i'} \sqrt{2}\cos(2 \pi n_\ell x) \\
= &  \frac{\lambda}{2} \sum_{m \geq 1} c_m \sum_{\ell\in \Delta_i'} \sqrt{2}\cos(2 \pi (m + n_\ell) x) + \frac{\lambda}{2} \sum_{m \geq 1} c_m \sum_{\ell \in \Delta_i'} \sqrt{2}\cos(2 \pi (m - n_\ell) x)
\end{align*}
In the sequel, let us write $i^*$ for the largest index in $\Delta_i'$. By construction we have $i^* \geq i^+ + s$, so that
\begin{equation} \label{n4n}
n_{i^*} + 4n_{i^+} \leq n_{i^*} \left( 1 + 4q^{-s} \right) \leq \sqrt{q} n_{i^*}
\end{equation}
by the second inequality in \eqref{s_size}. Accordingly, since the frequencies of $T_i(x)$ are by construction all between $n_{i^-}$ and $4 n_{i^+}$, the function
$$
\frac{\lambda}{2} \sum_{m \geq 1} c_m \sum_{\ell \in \Delta_i'} \sqrt{2}\cos(2 \pi (m + n_\ell) x)
$$
is a trigonometric cosine-polynomial whose frequencies are all between $n_{i^-}$ and $n_{i^*} + 4 n_{i^+} \leq \sqrt{q} n_{i^*}.$ From what was said in the case distinction above, the function 
$$
\frac{\lambda}{2} \sum_{m \geq 1} c_m \sum_{\ell \in \Delta_i'} \sqrt{2}\cos(2 \pi (m - n_\ell) x)
$$
is a trigonometric cosine-polynomial whose frequencies are all between $n_{i^-}$ and $n_{i^*}$, except for those arising from solutions of \eqref{mell}. Put formally, by the case distinction made above, there exists a trigonometric polynomial $U_i(x)$ whose frequencies are all between $n_{i^-}$ and $n_{i^*}$, such that
$$
T_i(x)\lambda \sum_{\ell \in \Delta_i'} \sqrt{2}\cos(2 \pi n_\ell x) = U_i (x) + \mathcal{O} \left(\lambda^3+ \lambda^3 L + \lambda^4 L^2 + \lambda^5 L^3 \right) = U_i (x) + \mathcal{O} \left(\lambda^3 L \right)
$$
(note how all error terms coming from the case distinction above are multiplied with another factor $\lambda$, coming from the multiplication of $T_i(x)$ with $\lambda \sum_{\ell \in \Delta_i'} \sqrt{2}\cos(2 \pi n_\ell x)$).\\

Accordingly, continuing from \eqref{continue_here}, and writing 
$$
W_i(x) := T_i(x) +  \lambda \sum_{k\in \Delta_i'} \sqrt{2}\cos(2 \pi n_k x)  + U_i(x),
$$
we have
$$
\exp \left(\lambda \sum_{k\in \Delta_i} \sqrt{2}\cos(2 \pi n_k x) \right)   \exp \left(\lambda \sum_{k\in \Delta_i'} \sqrt{2}\cos(2 \pi n_k x) \right) \leq 1 + \frac{\lambda^2 L}{2} + W_i(x) + \mathcal{O}(\lambda^3 L),
$$
where $W_i$ is a trigonometric cosine-polynomial all of whose frequencies are between $n_{i^-}$ and $\sqrt{q} n_{i^*}$. 
Plugging this into \eqref{hölder_this}, we have
$$
\int_0^1 \exp \left( \lambda \sum_{k=1}^N \sqrt{2}\cos(2 \pi n_k x) \right) \mathrm{d}x \leq \int_0^1 \prod_{i=1}^M \left( 1 + \frac{\lambda^2 L}{2} + W_i(x) + \mathcal{O}(\lambda^3 L) \right)  \mathrm{d}x.
$$
 Note that the reason for the whole construction was to ensure that integration and multiplication can be exchanged on the right-hand side of the previous equation. We claim that we have the equality
 \begin{align}
 \int_0^1 \prod_{i=1}^M \left( 1 + \frac{\lambda^2 L}{2} + W_i(x) \right) dx & =  \prod_{i=1}^M \int_0^1 \left( 1 + \frac{\lambda^2 L}{2} + W_i(x) \right) dx \label{see_this}\\
 & = \prod_{i=1}^M \left( 1 + \frac{\lambda^2 L}{2}\right), \nonumber
 \end{align}
 which then entails

 \begin{align*}
 \int_0^1 \prod_{i=1}^M \left( 1 + \frac{\lambda^2 L}{2} + W_i(x) + {\mathcal{O}(\lambda^3 L)} \right) dx & \le  \prod_{i=1}^M \int_0^1 \left( 1 + \frac{\lambda^2 L}{2} + W_i(x) + {\mathcal{O}(\lambda^3 L)} \right) dx \\
 & = \prod_{i=1}^M \left( 1 + \frac{\lambda^2 L}{2} + {\mathcal{O}(\lambda^3 L)}\right).
 \end{align*}
 Here we used the fact that, even though the $\mathcal{O}(\lambda^3 L)$ term also depends on $x$, we can still bound each factor $1 + \frac{\lambda^2 L}{2} + W_i(x) + \mathcal{O}(\lambda^3 L)$ from above. By the construction of $W_i(x)$, we have that each such factor is positive and thus a product of upper bounds is an upper bound for the product.
 
 To see the validity of \eqref{see_this}, i.e.\ the uncorrelatedness of the $W_i$, let us assume that after multiplying out in \eqref{see_this}, we want to integrate a product $W_{h_1}(x) W_{h_2}(x) \cdots W_{h_r}(x)$ for some $1 \leq h_1 < \dots < h_r \leq M$ and for some $r \leq M$. Since the $W_i$ are all sums of trigonometric functions, what we actually need to show is that
 $$
 \int_0^1 \prod_{i=1}^r \cos(2 \pi \omega_{h_i} x) dx = 0,
 $$
 whenever $\omega_{h_i}$ is a frequency that appears in $W_{h_i}(x)$, for $1 \leq i \leq r$. This is true if there is no configuration of $\pm$ signs such that
 $$
 \pm  \omega_{h_1} \pm \omega_{h_2} \pm \dots \pm \omega_{h_r} = 0.
 $$
 Recall that the $W_i$ were defined in such a way that for $1 \leq i \leq r$, the function $W_{h_i}$ is a sum of cosine-terms, all of whose frequencies are between $n_{h_i^-}$ and $\sqrt{q} n_{h_i^*}$ (in accordance with the definitions from before, in what follows $h_i^-$ denotes the smallest index in the block $\Delta_{h_i}$, while $h_i^+$ denotes the largest index in $\Delta_{h_i}$, and $h_i^*$ denotes the largest index in $\Delta_{h_i}'$). Thus we have in particular $\omega_{h_r} \geq n_{{h_r}^-}$, and $\omega_{h_{r-1}} \leq  \sqrt{q} n_{h_{r-1}^*}$. Furthermore, clearly $h_{r-1}^* < {h_r}^-$, so that by the lacunary growth condition we have $q n_{h_{r-1}^*} \leq n_{h_r^-}$. Accordingly, 
$$
\omega_{h_{r-1}} \leq \sqrt{q} n_{h_{r-1}^*} \leq \frac{n_{h_r^-}}{\sqrt{q}} \leq \frac{\omega_{h_r}}{\sqrt{q}},
$$
which implies that
\begin{equation} \label{omega_1}
 \omega_{h_r} - \omega_{h_{r-1}} \geq \left(1 - q^{-1/2} \right) \omega_{h_r} \geq \left(1 - q^{-1/2} \right) n_{{h_r}^-}.
\end{equation}
 
The sum of all the remaining frequencies is small in comparison. More precisely, by construction we have $\omega_{h_{r-2}} \geq \sqrt{q} \omega_{h_{r-3}}$, we have $\omega_{h_{r-3}} \geq \sqrt{q} \omega_{h_{r-4}}$, and so on, so that
\begin{eqnarray}
\omega_{h_{r-2}} + \omega_{h_{r-3}} + \dots + \omega_{h_1} & \leq & \omega_{h_{r-2}} \sum_{u=0}^\infty q^{-u/2} \nonumber\\
& \leq & \sqrt{q} n_{h_{r-2}^*} \left(1 - q^{-1/2} \right)^{-1}.  \label{omega_2}
\end{eqnarray}

Since ${{h_r}^-}$ exceeds ${h_{r-2}^*}$ by at least $L+s \geq 2s+1$, we have 
$$
 \left(1 - q^{-1/2} \right)n_{{h_r}^-} \geq  q^{2s+1}  \left(1 - q^{-1/2} \right) n_{h_{r-2}^*} > \sqrt{q}n_{h_{r-2}^*} \left(1 - q^{-1/2} \right)^{-1},
$$
which by \eqref{omega_1} and \eqref{omega_2} shows that $\pm  \omega_{h_1} \pm \omega_{h_2} \pm \dots \pm \omega_{h_r} = 0$ is indeed impossible. 

Thus, using $1 + \frac{\lambda^2 L}{2} + \mathcal{O}(\lambda^3 L)= \exp \left( \frac{\lambda^2 L}{2} + \mathcal{O}(\lambda^3 L) \right)$, we obtain

\begin{align}
\int_0^1 \exp \left( \lambda \sum_{k=1}^N \sqrt{2}\cos(2 \pi n_k x) \right) \mathrm{d}x \leq & \prod_{i=1}^M \left( 1 + \frac{\lambda^2 L}{2}  + \mathcal{O} (\lambda^3 L)\right) \nonumber\\
= & \prod_{i=1}^M \exp \left( \frac{\lambda^2 L}{2} + \mathcal{O}(\lambda^3 L) \right) \nonumber\\
=& \exp \left( \frac{\lambda^2 N}{2}  + \mathcal{O} (\lambda^3 N)\right), \label{MGF_upper}
\end{align}
as claimed.\\

Now we come to the lower bound, which is very similar but a bit easier. Starting again at \eqref{hölder_this}, and using now that $e^x \geq 1 + x + x^2/2 + x^3/6$ and $e^x \geq 1+x$ (and noting that $1+x+x^2/2+x^3/6$ and $1+x$ are both non-negative for $|x|\leq 1$), we have
\begin{eqnarray*}
& & \int_0^1 \exp \left( \lambda \sum_{k=1}^N \sqrt{2}\cos(2 \pi n_k x) \right) \mathrm{d}x \\
& \geq & \int_0^1 \prod_{i=1}^M \left( \left(1 +  \lambda \sum_{k\in \Delta_i} \sqrt{2}\cos(2 \pi n_k x)+\frac{\lambda^2}{2} \left(\sum_{k\in \Delta_i}\sqrt{2}\cos(2 \pi n_k x)\right)^2 \right. \right. \\
& & \left.\left. \qquad  + \frac{\lambda^3}{6} \left(\sum_{k\in \Delta_i}\sqrt{2}\cos(2 \pi n_k x)\right)^3 \right) \cdot \left(1 +  \lambda \sum_{k\in \Delta_i'} \sqrt{2}\cos(2 \pi n_k x) \right) \right)\mathrm{d}x.
\end{eqnarray*}
Now we can carry out an analysis which is analogous to the one for the upper bound, but slightly easier since a) we do not have the fourth-order term coming from the polynomial approximation, and b) after multiplying out we can safely ignore all terms where $\lambda$  appears in second or fourth power, since after integration the contribution coming from these terms is always non-negative. In this way we can establish the lower bound 

$$
\int_0^1 \exp \left( \lambda \sum_{k=1}^N \sqrt{2}\cos(2 \pi n_k x) \right) \mathrm{d}x \geq  \exp \left( \frac{\lambda^2 N}{2}  + \mathcal{O} (\lambda^3 N)\right).
$$

Together with \eqref{MGF_upper}, we have obtained matching upper and lower bounds, so that we finally arrive at
$$
\int_0^1 \exp \left( \lambda \sum_{k=1}^N \sqrt{2}\cos(2 \pi n_k x) \right) \mathrm{d}x = \exp \left(\frac{\lambda^2  N}{2} + \mathcal{O}(\lambda^3 N) \right),
$$
as desired.
\end{proof}

\begin{proof}[Proof of Corollary \ref{cor:MDP}]
We apply \cite[Theorem 2.3.6]{DZLDPs} together with the subsequent remark. Let $t \in \mathbb{R}$ and $\lambda = \lambda_N$ be a null sequence with $\sqrt{N} \lambda \to \infty$. Then, by Theorem \ref{th1} it follows that
\begin{align*}
\lim_{N \rightarrow \infty} \lambda^2 \log \mathbb{E} \left[ e^{t \frac{\lambda}{\sqrt{N} \lambda^2} \sum_{k=1}^N \sqrt{2} \cos(2 \pi n_k x) }\right] & = \frac{t^2}{2}.
\end{align*}
The latter function is clearly essentially smooth and thus \cite[Theorem 2.3.6]{DZLDPs} implies the claim by recalling that $t \mapsto \frac{t^2}{2}$ is a fixed point of the Legendre transform.
\end{proof}

\bibliography{Lacunary_MGF}
\bibliographystyle{abbrv}

\end{document}